\documentclass[a4paper,leqno,10pt]{amsart}

\usepackage{epsf,graphicx,epsfig}
\usepackage{amscd}
\usepackage{amsmath,latexsym,amssymb,amsthm}
\usepackage[nospace,noadjust]{cite}
\usepackage{textcomp}
\usepackage{setspace,cite}
\usepackage{lscape,fancyhdr,fancybox}
\usepackage{stmaryrd}
\usepackage[all,cmtip]{xy}
\RequirePackage{tikz-cd}
\RequirePackage{amssymb}
\usetikzlibrary{calc}
\usetikzlibrary{decorations.pathmorphing}

\tikzset{curve/.style={settings={#1},to path={(\tikztostart)
    .. controls ($(\tikztostart)!\pv{pos}!(\tikztotarget)!\pv{height}!270:(\tikztotarget)$)
    and ($(\tikztostart)!1-\pv{pos}!(\tikztotarget)!\pv{height}!270:(\tikztotarget)$)
    .. (\tikztotarget)\tikztonodes}},
    settings/.code={\tikzset{quiver/.cd,#1}
        \def\pv##1{\pgfkeysvalueof{/tikz/quiver/##1}}},
    quiver/.cd,pos/.initial=0.35,height/.initial=0}

\tikzset{tail reversed/.code={\pgfsetarrowsstart{tikzcd to}}}
\tikzset{2tail/.code={\pgfsetarrowsstart{Implies[reversed]}}}
\tikzset{2tail reversed/.code={\pgfsetarrowsstart{Implies}}}
\tikzset{no body/.style={/tikz/dash pattern=on 0 off 1mm}}

\usepackage{graphicx}

\usepackage{color}
\usepackage{url}
\usepackage{enumerate}
\usepackage[mathscr]{euscript}

\newcounter{cnt1}
\newcounter{cnt2}
\newcounter{cnt3}
\newcommand{\blr}{\begin{list}{$($\roman{cnt1}$)$} {\usecounter{cnt1}
        \setlength{\topsep}{0pt} \setlength{\itemsep}{0pt}}}
\newcommand{\bla}{\begin{list}{$($\alph{cnt2}$)$} {\usecounter{cnt2}
        \setlength{\topsep}{0pt} \setlength{\itemsep}{0pt}}}
\newcommand{\bln}{\begin{list}{$($\arabic{cnt3}$)$} {\usecounter{cnt3}
                \setlength{\topsep}{0pt} \setlength{\itemsep}{0pt}}}
\newcommand{\el}{\end{list}}

\newtheorem{Thm}{Theorem}[section]

\newtheorem{Prop}[Thm]{Proposition}
\newtheorem{Def}[Thm]{Definition}
\newtheorem{Exm}[Thm]{Example}
\newtheorem{Rem}[Thm]{Remark}

\title{}
\author{}
\date{}
\usepackage{amssymb}

\usepackage{hyperref}
\hypersetup{
	colorlinks,
	citecolor=blue,
	filecolor=black,
	linkcolor=blue,
	urlcolor=black
}

\begin{document}
\title{Poisson dialgebras}

\author{Apurba Das}
\address{Department of Mathematics,
Indian Institute of Technology Kharagpur, Kharagpur-721302, West Bengal, India.}
\email{apurbadas348@gmail.com, apurbadas348@maths.iitkgp.ac.in}

\author{Satyendra Kumar Mishra}\footnote{Corresponding author email [A2]: satyamsr10@gmail.com; satyendra.mishra@tcgcrest.org}

\author{Goutam Mukherjee}
\address{TCG Centres for Research and Education in Science and Technology, Institute for Advancing Intelligence, Salt Lake, Kolkata-700091, West Bengal, India.}
\address{Academy of Scientific and Innovative Research (AcSIR), Ghaziabad- 201002, India.}
\email{satyamsr10@gmail.com, goutam.mukherjee@tcgcrest.org, gmukherjee.isi@gmail.com}



\begin{abstract}
The notion of Poisson dialgebras was introduced by Loday. In this article, we propose a new definition with some modifications that is supported by several canonical examples coming from Poisson algebra modules, averaging operators on Poisson algebras, and differential Poisson algebras. We show that a Poisson object in the category of linear maps has an associated Poisson dialgebra structure. Conversely, starting from a Poisson dialgebra we describe a Poisson object in the category of linear maps. These constructions yield a pair of adjoint functors between the category of Poisson objects in the category of linear maps and the category of Poisson dialgebras. There is a Lie $2$-algebra associated with any Leibniz algebra. Here, we first obtain an associative $2$-algebra starting from a dialgebra. Then, for a Poisson dialgebra, we construct a graded space that inherits both a Lie $2$-algebra and an associative $2$-algebra structure. In a particular case of Poisson dialgebras, which we call `reduced Poisson dialgebra', we obtain an associated $2$-term homotopy Poisson algebra (of degree $0$).  
\end{abstract}


\maketitle





\noindent

\thispagestyle{empty}



\section{Introduction}
The notion of Leibniz algebras was first introduced by Bloh \cite{Bloh}. A right Leibniz algebra is a vector space $L$ equipped with a bracket $[-,-]$ such that the map $[-,x]$ is a derivation of the bracket $[-,-]$ for each $x\in L$. Leibniz algebras are non-commutative analog of Lie algebras, which appear in different areas such as Poisson geometry, representation theory of Lie algebras, and Nambu-Mechanics.

\smallskip

A dialgebra is a vector space equipped with two associative algebra products satisfying three compatibility relations. Dialgebras were also introduced by Loday \cite{loday}. Dialgebras play a dominant role in the theory of Leibniz algebras generalizing the role of associative algebra in the theory of Lie algebras. The notion of dialgebras is a non-commutative analog of associative algebras, in particular, we have the following diagram.

\[\begin{tikzcd}
	{\text{Lie algebras}} &&&& {\text{Associative algebras}} \\
	\\
	{\text{Leibniz algebras}} &&&& {\text{Dialgebras}}
	\arrow["{\text{non-commutative analog}}"', shorten >=3pt, from=1-1, to=3-1]
	\arrow["{\text{non-commutative analog}}", from=1-5, to=3-5]
	\arrow["{\text{enveloping algebra functor}}", shift left=1, from=1-1, to=1-5]
	\arrow["{\text{$[x,y]=xy-yx$}}", shift left=1, from=1-5, to=1-1]
	\arrow["{\text{eneveloping algebra functor}}", shift left=1, from=3-1, to=3-5]
	\arrow["{\text{$[x,y]=x\dashv y-y\vdash x$}}", shift left=1, from=3-5, to=3-1]
\end{tikzcd}\]

\smallskip

Loday and Pirashvili considered a tensor category of linear maps $\mathcal{LM}$ in order to prove a Milnor-Moore type theorem for Leibniz algebras in \cite{Loday-Pirashvili}. They showed that any Leibniz algebra is a Lie object in the category of linear maps. The universal enveloping algebra functor from the category of Leibniz algebras to the category of associative algebras factors through the category of cocommutative Hopf algebras in $\mathcal{LM}$. Here, we first show that there is an adjunction between the category of associative algebra objects in the category $\mathcal{LM}$ and the category of dialgebras. Thus, the notion of dialgebras is a non-commutative analogue of associative algebras in the sense of \cite{Loday-Pirashvili}.

\smallskip

Poisson algebras are central objects in Poisson geometry. Throughout the article, by Poisson algebra we mean a vector space equipped with a Lie bracket and an associative algebra product such that the Lie bracket is a derivation of the algebra product, we refer to \cite{Akman, YKS, Kubo} for more details. The main goal of this article is to study a non-commutative analog of Poisson algebras (in the sense of \cite{Loday-Pirashvili}) known as \textit{Poisson dialgebra} first introduced by Loday in \cite{loday}. We redefine Poisson dialgebras consisting of a Leibniz algebra and a dialgebra structure satisfying some compatibility relations (see Definition \ref{Pois-di}). We add two conditions on mixed skew-symmetry to the definition of Loday \cite{loday}. In a sequel, we justify the modifications in the definition with several canonical examples coming from dialgebras, Poisson bimodules, averaging operators, and filtered dialgebras. Later on, we establish an adjunction between the category of poisson dialgebras and Poisson algebra objects in $\mathcal{LM}$.


\smallskip

The notion of Lie $2$-algebras was introduced by Baez and Crans as the categorification of Lie algebras \cite{Baez}. They showed that Lie $2$-algebras are equivalent to $2$-term  homotopy Lie algebras ($2$-term $L_\infty$-algebras \cite{Lada}). It has been observed in \cite{Sheng} that Leibniz algebras are closely related to Lie 2-algebras. Homotopy algebras such as homotopy associative algebras, homotopy Lie algebras, and homotopy Poisson algebras have been extensively studied, we refer to \cite{Baez, Lada, HomotopyPoiss, Mehta, PoissManifold1,HigherPoiss1, HigherPoiss2, stas} to mention some of these studies. A $2$-term homotopy Poisson algebra  is a Lie $2$-algebra endowed with an associative multiplication satisfying some compatibilities, which in turn gives a homotopy Poisson algebra of finite type \cite{HomotopyPoiss, Mehta}. We first show that a dialgebra naturally gives rise to an associative $2$-algebra structure. This can be thought of as the dialgebra analogue of the construction given in \cite{Sheng}. Next, we will further generalize the construction for Poisson dialgebras. Given a Poisson dialgebra, we construct a graded space that inherits both a Lie $2$-algebra and an associative $2$-algebra structure. Additionally, if the given Poisson dialgebra is reduced (i.e., a Poisson dialgebra with $x\dashv y= x\vdash y,$ for all $x,y\in P$), then the induced associative $2$-algebra structure becomes associative and satisfies all the compatibilities of a $2$-term homotopy Poisson algebra. 

\smallskip

The organization of the paper is as follows: In Section \ref{sec-2}, we recall the definitions of Poisson algebras, dialgebras, Leibniz algebras, tensor category of linear maps and related results. In Section \ref{sec-3}, we construct an associative $2$-algebra from a dialgebra. We propose a new definition of Poisson dialgebras in Section \ref{sec-4} followed by different natural examples. We show that dialgebras, Poisson algebra bimodules, differential Poisson algebras, and averaging operators on Poisson algebras give a Poisson dialgebra structure. We also define a graded Poisson dialgebra associated to a filtered dialgebra. We describe two adjoint functors between the categories of Poisson dialgebras and Poisson algebra objects in the category $\mathcal{LM}$. The last section is devoted to homotopy structures related to Poisson dialgebras. For a Poisson dialgebra $(P,\dashv,\vdash,[-,-])$, we define a graded space $P\oplus J$ that simultaneously carries a Lie $2$-algebra and an associative $2$-algebra structure. In the particular case of a reduced Poisson dialgebra, we obtain a $2$-term homotopy Poisson algebra (degree $0$) on the graded space $P\oplus J$.

\section{Preliminaries}\label{sec-2}

In this section, we recall some definitions and fix notations. Let $\mathbb{K}$ be a field of characteristic zero. All the linear maps and tensor products are over the field $\mathbb{K}$ unless otherwise stated. 
Let $(A,\mu)$ is an associative algebra and $M$ be an $A$-bimodule. Then, we denote $ab:=\mu(a,b)$ for all $a,b\in A$. We denote the left and right module actions of $A$ on $M$ by $a\cdot m$ and $m\cdot a$, respectively for all $a\in A$ and $m\in M$.

\subsection{Poisson algebras}

A Poisson algebra is a $\mathbb{K}$-vector space $P$ equipped with an associative algebra product $\mu$ and a Lie bracket $[-,-]$ such that the following compatibility relation holds.
\begin{equation}\label{Leibniz rule}
[x, yz]=y[x,z]+[x,y]z, \quad \forall x,y,z\in P.
\end{equation}

A map  $\varphi:P\rightarrow P^\prime$ is called a homomorphism of Poisson algebras if for all $x,y\in P$, 
$$\varphi(xy)=\varphi(x)\varphi(y)\quad\mbox{and}\quad \varphi[x,y]_P=[\varphi(x),\varphi(y)]_{P^\prime}.$$ We denote the category of Poisson algebras by $\mathsf{Pois}$. 

\begin{Def}
Let $(P,\mu,[-,-])$ be a Poisson algebra. Then a $\mathbb{K}$-vector space is called a Poisson-bimodule over $P$ if $M$ is simultaneously an associative algebra bimodule over $(P,\mu)$ and a Lie algebra module over $(P,[-,-])$ such that the compatibility relation \eqref{Leibniz rule} holds whenever two of the entries are in $P$ and one of the entries is in $M$. 
\end{Def}

\subsection{Associative dialgebras}
An associative dialgebra $(D,\dashv, \vdash)$ is a vector space $D$ equipped with two linear maps $\dashv,\vdash:D\otimes D\rightarrow D$ (called the left product and the right product, respectively) such that the following identities hold
\begin{equation}\label{ax1}
(x\dashv y) \dashv z = x\dashv (y\vdash z),
\end{equation}
\vspace{-.8cm}
\begin{equation}\label{ax2}
(x\dashv y) \dashv z = x\dashv (y\dashv z),
\end{equation}
\vspace{-.8cm}
\begin{equation}\label{ax3}
(x\vdash y) \dashv z = x\vdash (y\dashv z),
\end{equation}
\vspace{-.8cm}
\begin{equation}\label{ax4}
(x\dashv y) \vdash z = x\vdash (y\vdash z),
\end{equation}
\vspace{-.8cm}
\begin{equation}\label{ax5}
(x\vdash y) \vdash z = x\vdash (y\vdash z),
\end{equation}
for all $x,y,z\in D$. Note that identities \eqref{ax2} and \eqref{ax5} imply that the left and right products are associative. Thus, a dialgebra is a vector space $D$ with two associative products $\dashv$ and $\vdash$ satisfying 
\begin{equation}\label{ax6}
x\dashv (y \dashv z) = x\dashv (y\vdash z),
\end{equation}
\vspace{-.8cm}
\begin{equation}\label{ax7}
(x\vdash y) \dashv z = x\vdash (y\dashv z),
\end{equation}
\vspace{-.8cm}
\begin{equation}\label{ax8}
(x\dashv y) \vdash z = (x\vdash y)\vdash z),
\end{equation}
for all $x,y,z \in D$. A linear map $f:D\rightarrow D^\prime$ is called a homomorphism of dialgebras if $f(x\dashv y)=f(x)\dashv f(y)$
and $f(x\vdash y)=f(x)\vdash f(y),$
for all $x,y\in D$. We denote the category of associative dialgebras by $\mathsf{Dias}$.

We recall some interesting examples of dialgebras.
\begin{itemize}
\item Let $A$ be an associative algebra, then $(A,\dashv,\vdash)$ is a dialgebra, where $a\dashv b=ab=a\vdash b,~\forall a,b\in A$.

\item Let $A$ be an associative algebra and $M$ be an $A$-bimodule. If $f: M\rightarrow A$ is an $A$-bimodule map, then $(M,\dashv,\vdash)$ is a dialgebra, where the left and right products are given by 
$$m\dashv n=m\cdot f(n),\quad \mbox{and} \quad m\vdash n=f(m)\cdot n,~~~ \forall m,n\in M.$$

\item Let $A$ be an associative algebra and $d:A\rightarrow A$ be a linear map. If $(A,d)$ is a differential associative algebra, i.e., $d(ab)=(da)b+a(db),$ for all $a,b\in A$. Then $(A,\dashv,\vdash)$ is a dialgebra, where the left and right products are given by 
$$a\dashv b=a d(b),\quad \mbox{and} \quad a\vdash b=d(a) b,~~~ \forall a,b\in A.$$

\end{itemize}

\begin{Rem}\label{associativization}
Let $\mathsf{As}$ denotes the category of associative algebras. Any associative algebra is a dialgebra. Thus, we have a functor $\mathsf{inc}: \mathsf{As}\rightarrow \mathsf{Dias}$. For a dialgebra $(D,\dashv,\vdash)$, let us consider the quotient space $D_{As}$ of $D$ by the subspace spanned by the elements $x\dashv y-x\vdash y,~\forall x,y\in D$. Then $D_{As}$ is an associative algebra with the product induced by the products $\dashv,\vdash: D\otimes D\rightarrow D$. It yields a functor $\mathsf{(-)_{As}}:\mathsf{Dias}\rightarrow \mathsf{As}$ called associativization functor. The functor $\mathsf{(-)_{As}}$ is left adjoint to the functor $\mathsf{inc}: \mathsf{As}\rightarrow \mathsf{Dias}$.
\end{Rem}

A $\mathbb{K}$-vector space $M$ is called a left module over a dialgebra $(D, \dashv,\vdash)$ if there are linear maps $\dashv,\vdash : D\otimes M \rightarrow M$ satisfying the equations \eqref{ax1}-\eqref{ax5} with two of the entries in $D$ and one of the entries in $M$. Similarly, we can define the notion of right modules over $(D, \dashv,\vdash)$.

\begin{Def}\label{rep-dialgebra}

A $\mathbb{K}$-vector space $M$ is a bimodule over a dialgebra $(D, \dashv,\vdash)$ if there are linear maps $\dashv, \vdash: M\otimes D\rightarrow M$ and $\dashv, \vdash: D\otimes M\rightarrow M$ such that the equations \eqref{ax1}-\eqref{ax5} hold true with two of the entries in $D$ and one of the entries in $M$.
\end{Def}

\subsection{Leibniz algebras}

A (right) Leibniz algebra $(\mathfrak{g},[-,-])$ is a $\mathbb{K}$-vector space equipped with a linear map $[-,-]:\mathfrak{g}\otimes \mathfrak{g}\rightarrow \mathfrak{g}$ (bracket) satisfying the Leibniz identity
$$[[x,y],z]=[[x,z],y]+[x,[y,z]],\quad \forall x,y,z\in \mathfrak{g}.$$

If we consider the quotient of $\mathfrak{g}$ by the ideal generated by the elements of the form $[x,x],~x\in \mathfrak{g}$, then it gives a Lie algebra $\mathfrak{g}_{Lie}$ associated to the Leibniz algebra $(\mathfrak{g},[-,-])$. 

Let us also recall the following result from \cite{loday}.
\begin{Prop}[\cite{loday}]\label{induced-Leibniz}
 For a dialgebra $(D,\dashv,\vdash)$, the pair $(D,[-,-])$ is a Leibniz algebra with the bracket
$$[x,y]:=x\dashv y-y\vdash x,\quad \forall x,y\in D.$$
\end{Prop}
\begin{proof}
It is straightforward to check that $(D,[-,-])$ is a Leibniz algebra.
\end{proof}

Furthermore, it easily follows that the above result gives a functor from the category $\mathsf{Dias}$ to the category $\mathsf{Leib}$ of Leibniz algebras.

\subsection{The category of linear maps }
The category $\mathcal{LM}$ of linear maps is defined as follows:
\begin{itemize}
\item The objects are the $\mathbb{K}$-linear maps $f:V\rightarrow W$, where $V$ and $W$ are $\mathbb{K}$-vector spaces.

\item A morphism $(\alpha^\prime,\alpha)$ from $V\xrightarrow{f} W$ and $V^\prime\xrightarrow{f^\prime} W^\prime$ is a commutative diagram in the category of vector spaces, given by
\begin{equation*}\label{d1}
\begin{CD}
 V @>\alpha^\prime >> V^\prime   \\
 @V f VV @VV f^\prime V\\
 W @>>\alpha > W^\prime.
\end{CD}
\end{equation*}

\end{itemize}

For an object $V\xrightarrow{f} W$, by upstairs we mean the vector space $V$ and by downstairs we mean the vector space $W$. If $V\xrightarrow{f} W$ and $V^\prime\xrightarrow{f^\prime} W^\prime$ be two objects in the category $\mathcal{LM}$, then a tensor product of these two objects can be defined as follows:
$$f\otimes_{\mathcal{LM}} f^\prime: V\otimes_{\mathbb{K}} W^\prime \oplus W\otimes_{\mathbb{K}} V^\prime\rightarrow W\otimes_{\mathbb{K}} W^\prime,$$
where 
$$f\otimes_{\mathcal{LM}} f^\prime:=f\otimes_{\mathbb{K}} \mathsf{1}+ \mathsf{1}\otimes_{\mathbb{K}}f^\prime.$$

The tensor product of two morphisms $(\alpha^\prime,\alpha)$ and $(\beta^\prime, \beta)$ is given by 
$$(\alpha^\prime,\alpha)\otimes_{\mathcal{LM}}(\beta^\prime, \beta):= (\alpha^\prime\otimes_{\mathbb{K}}\beta+\alpha \otimes_{\mathbb{K}} \beta^\prime,\alpha \otimes_{\mathbb{K}}\beta).$$

\begin{Def}
An associative algebra object in the category $\mathcal{LM}$ is an object $M\xrightarrow{f} A$ with a morphism 
$$\bar{\mu}=(\mu^\prime,\mu):(M\xrightarrow{f} A)\otimes_{\mathcal{LM}}(M\xrightarrow{f} A)\rightarrow (M\xrightarrow{f} A)$$
such that
\begin{equation}\label{associativitiy-in-LM}
\bar{\mu}(\mathsf{1}\otimes_{\mathcal{LM}}\bar{\mu})=\bar{\mu}(\bar{\mu}\otimes_{\mathcal{LM}}\mathsf{1}).
\end{equation}
\end{Def}

The downstairs of the above condition \eqref{associativitiy-in-LM} implies that $A$ is an associative algebra with the product $\mu$. The upstairs implies that the map $\mu^\prime$ yields an $A$-bimodule structure on $M$. Moreover, since $\bar{\mu}=(\mu^\prime,\mu)$ is a morphism in $\mathcal{LM}$, we have
$$f(\mu^\prime(m,a))=\mu(f(m),a),~\quad\mbox{and } f(\mu^\prime(a,m))=\mu(a,f(m)),\quad \forall m\in M, a\in A.$$

Therefore, we have the following result.

\begin{Prop}
An associative algebra object $M\xrightarrow{f} A$ in the category $\mathcal{LM}$ is equivalent to an $A$-bimodule map $f:M\rightarrow A$, where $A$ is an associative algebra and $M$ is an $A$-bimodule. \qed
\end{Prop}

Let $M\xrightarrow{f} A$ be an associative algebra in the category $\mathcal{LM}$. Then we have a dialgebra structure $(M,\dashv,\vdash)$, where the products $\dashv, \vdash$ are defined by 
\begin{equation}\label{diproducts}
m\dashv n:=m\cdot f(n) \quad \mbox{and}\quad m\vdash n:=f(m)\cdot n,\quad \forall m,n\in M.
\end{equation}

\begin{Prop}
If $M\xrightarrow{f} A$ is an associative algebra object in the category $\mathcal{LM}$. Then, $(M,\dashv,\vdash)$ is a dialgebra and $f: M\rightarrow A$ is a dialgebra homomorphism.
\end{Prop}
\begin{proof}
It is easy to check that the products $\dashv, \vdash$ defined by equation \eqref{diproducts} are associative and satisfy equations \eqref{ax6}-\eqref{ax8}. Therefore, $(M,\dashv, \vdash)$ is a dialgebra.

Moreover, since $f: M\rightarrow A$ is an $A$-bimodule map, we have $f(m\dashv n)=f(m)f(n)$ and $f(m\vdash n)=f(m) f(n)$ for all $m,n\in M$. Thus, $f: M\rightarrow A$ is a dialgebra homomorphism. 
\end{proof}

\begin{Def}

A homomorphism between two associative algebra objects $(M\xrightarrow{f} A)$ and $(M^\prime\xrightarrow{f^\prime } A^\prime)$ (with the products $\bar{\mu}$ and $\bar{\mu}^\prime$, respectively) in $\mathcal{LM}$ is
a pair $\Phi=(\phi^\prime,\phi):(M\xrightarrow{f} A)\rightarrow(M^\prime\xrightarrow{f^\prime } A^\prime)$ such that $$\bar{\mu}^\prime\circ(\Phi\otimes_{\mathcal{LM}}\Phi)=\Phi\circ\bar{\mu}$$ 
and the following diagram is commuative 

\begin{equation}\label{d2}
\begin{CD}
 M @>\phi^\prime >> M^\prime   \\
 @V f VV @VV f^\prime V\\
 A @>>\phi > A^\prime.
\end{CD}
\end{equation}

\end{Def}

Thus, an associative algebra homomorphism between $(M\xrightarrow{f} A)$ and $(M^\prime\xrightarrow{f^\prime } A^\prime)$ is a pair $(\phi^\prime,\phi)$ of a dialgebra homomorphism $\phi^\prime: M\rightarrow M^\prime$ and an associative algebra homomorphism $\phi: A\rightarrow A^\prime$ such that the diagram \eqref{d2} commutes.

Let us denote the category of associative algebra objects in $\mathcal{LM}$ by $\mathsf{As}_{\mathcal{LM}}$. 
Then, we have a functor $\mathsf{F}: \mathsf{As}_{\mathcal{LM}}\rightarrow \mathsf{Dias}$ assigning a dialgebra $M$ for an associative algebra object $M\xrightarrow{f} A$ in $\mathcal{LM}$. On the other hand, we have a functor $\mathsf{G}: \mathsf{Dias}\rightarrow \mathsf{As}_{\mathcal{LM}}$ defined as follows: For a dialgebra $D$, the functor $\mathsf{G}$ sends the dialgebra $D$ to the associative algebra object $D\xrightarrow{p} D_{As}$ where $D_{As}$ is the associative algebra associated to the dialgebra $D$ and $p$ is the projection map.

\begin{Thm}\label{As_LM}
With the above notations, the functor $\mathsf{G}:\mathsf{Dias}\rightarrow \mathsf{As}_{\mathcal{LM}}$ is left adjoint to the functor $\mathsf{F}: \mathsf{As}_{\mathcal{LM}}\rightarrow \mathsf{Dias}$, i.e., there is a bijection
$$\mathrm{Hom}_{\mathsf{As}_{\mathcal{LM}}}(\mathsf{G}(D),M\xrightarrow{f} A)\cong \mathrm{Hom}_{\mathsf{Dias}}(D,\mathsf{F}(M\xrightarrow{f} A)),$$
for all $D\in \mathsf{Dias}$ and $M\xrightarrow{f}A\in \mathsf{As}_{\mathcal{LM}}$ such that the bijection is functorial in $D\in \mathsf{Dias}$ and $M\xrightarrow{f}A\in \mathsf{As}_{\mathcal{LM}}$.
\end{Thm}

\begin{proof}
Let $M\xrightarrow{f}A$ be an associative algebra object in ${\mathcal{LM}}$. Then, there is an induced dialgebra structure on $M$. Let $\phi^\prime \in \mathrm{Hom}_{\mathsf{Dias}}(D,\mathsf{F}(M\xrightarrow{f} A))$, i.e., we have a dialgebra homomorphism $\phi^\prime: D\rightarrow M$. Let us recall that the associativization functor $(-)_{As}:\mathsf{Dias}\rightarrow \mathsf{As}$ is left adjoint to the functor $\mathsf{inc}: \mathsf{As}\rightarrow \mathsf{Dias}$. Also, $f\circ \phi^\prime: D\rightarrow \mathsf{inc}(A)$ is a dialgebra homomorphism, so  there exists a unique algebra homomorphism $\phi: D_{As}\rightarrow A$ such that the following diagram commutes.
\begin{equation*}
\begin{CD}
 D @>\phi^\prime >> M   \\
 @V p VV @VV f V\\
 D_{As} @>>\phi > A.
\end{CD}
\end{equation*}
It follows that for any morphism $\phi^\prime \in \mathrm{Hom}_{\mathsf{Dias}}(D,\mathsf{F}(M\xrightarrow{f} A))$, there exists a unique map $(\phi^\prime,\phi)\in \mathsf{Hom}_{\mathsf{As}_{\mathcal{LM}}}(\mathsf{G}(D),M\xrightarrow{f} A)$.

Conversely, for any associative algebra morphism $(\phi^\prime,\phi)\in \mathrm{Hom}_{\mathsf{As}_{\mathcal{LM}}}(\mathsf{G}(D),M\xrightarrow{f} A)$, we have a dialgebra homomorphism $\phi\in \mathrm{Hom}_{\mathsf{Dias}}(D,\mathsf{F}(M\xrightarrow{f} A))$. Again, the uniqueness follows from the fact that the associativization functor $\mathsf{(-)_{As}}:\mathsf{Dias}\rightarrow \mathsf{As}$ is left adjoint to the functor $\mathsf{inc}: \mathsf{As}\rightarrow \mathsf{Dias}$. 
Hence, we have a bijection between the sets of homomorphisms
$$\mathrm{Hom}_{\mathsf{As}_{\mathcal{LM}}}(\mathsf{G}(D),M\xrightarrow{f} A)\cong \mathrm{Hom}_{\mathsf{Dias}}(D,\mathsf{F}(M\xrightarrow{f} A)).$$

For every morphism $\alpha:D\rightarrow D^\prime$ in $\mathsf{Dias}$ and $\beta:(M\xrightarrow{f}A)\rightarrow (M^\prime\xrightarrow{f^\prime}A^\prime)$ in $\mathsf{As}_{\mathcal{LM}}$, we have the following commutative diagram:

$$\begin{CD}
\mathrm{Hom}_{\mathsf{As}_{\mathcal{LM}}}(\mathsf{G}(D^{\prime}), M\xrightarrow{f}A) @> \mathsf{G}(\alpha)^* >>
\mathrm{Hom}_{\mathsf{As}_{\mathcal{LM}}}(\mathsf{G}(D), M\xrightarrow{f}A) @> \beta_{*} >> \mathrm{Hom}_{\mathsf{As}_{\mathcal{LM}}}(\mathsf{G}(D), M^{\prime}\xrightarrow{f^{\prime}}A^{\prime})\\
@V\cong VV @V\cong VV @V\cong VV\\
\mathrm{Hom}_{\mathsf{Dias}}(D^{\prime},\mathsf{F} (M\xrightarrow{f}A)) @> \alpha^* >>
\mathrm{Hom}_{\mathsf{Dias}}(D, \mathsf{F}(M\xrightarrow{f}A)) @> \mathsf{F}(\beta)_{*} >> \mathrm{Hom}_{\mathsf{Dias}}(D, \mathsf{F}(M^{\prime}\xrightarrow{f^{\prime}}A^{\prime})),
\end{CD}$$
where $\alpha^*$ and $\beta_*$ are the maps induced by $\alpha$ and $\beta$. Thus, the bijection is functorial in $D\in \mathsf{Dias}$ and $M\xrightarrow{f}A\in \mathsf{As}_{\mathcal{LM}}$.
\end{proof}

We denote the category of Lie objects in $\mathcal{LM}$ by $\mathsf{Lie_{\mathcal{LM}}}$. Then let us recall the following result for Lie algebra objects in the category of linear maps from \cite{Loday-Pirashvili}. 

\begin{Thm}
For a Lie algebra object $M\xrightarrow{f} L$ in $\mathcal{LM}$, we have a Leibniz algebra structure on $M$. On the other hand, for any Leibniz algebra $\mathfrak{g}$ the Liezation $\mathfrak{g}\xrightarrow{p}\mathfrak{g}_{Lie}$ is a Lie algebra object in $\mathcal{LM}$. The two functors $\mathsf{Leib}\longleftrightarrow \mathsf{Lie}_{\mathcal{LM}}$ are adjoint to each-other. 
\end{Thm}

\begin{proof}
The proof is similar to the proof of the Theorem \ref{As_LM}. For more details, see \cite{Loday-Pirashvili}.
\end{proof}
 
 \section{An associative $2$-algebra associated to a dialgebra}\label{sec-3}
 
 Let us recall that associative $2$-algebras are categorification of associative algebras. Associative $2$-algebras are equivalent to $2$-term $A_{\infty}$-algebras \cite{stas}. 
 
\begin{Def}
A graded vector space $\mathcal{A}=A_{0}\oplus A_{1}$ is called an associative $2$-algebra, if $\mathcal{A}$ is equipped with linear maps ${\mu_i:\otimes^i\mathcal{A}\rightarrow \mathcal{A}}$  of degrees $i-2$, for $i=1,2,3$ and the following identities hold.
\begin{enumerate}[(i)]
\item $\mu_1\mu_2(x,a)=\mu_2(x,\mu_1(a)),~~\mu_1\mu_2(a,x)=\mu_2(\mu_1(a),x)$,
\item $\mu_2(\mu_1(a),b)=\mu_2(a,\mu_1(b))$,
\item $\mu_2(\mu_2(x,y),z)-\mu_2(x,\mu_2(y,z))=\mu_1\mu_3(x,y,z)$,
\item $\mu_2(\mu_2(a,x),y)-\mu_2(a,\mu_2(x,y))=\mu_3(\mu_1(a),x,y)$,
\item $\mu_2(\mu_2(x,a),y)-\mu_2(x,\mu_2(a,y))=\mu_3(x,\mu_1(a),y)$,
\item $\mu_2(\mu_2(x,y),a)-\mu_2(x,\mu_2(y,a))=\mu_3(x,y,\mu_1(a))$,
\item $\mu_3(\mu_2(x,y),z,w)-\mu_3(x,\mu_2(y,z),w)+\mu_3(x,y,\mu_2(z,w))=\mu_2(\mu_3(x,y,z),w)+\mu_2(x,\mu_3(y,z,w))$,
\end{enumerate}
for all $a,b\in A_0$ and $x,y,z,w\in A_1$. 
\end{Def}
 
Let $(D,\dashv,\vdash)$ be a dialgebra. We consider the following subset of $D$
$$I:=\{x\in D|~~ x\vdash y=0=y\dashv x, \quad \forall y\in D\}.$$
Note that $I$ is an ideal in $D$, i.e., $x\dashv y\in I$ whenever atleast one of $x$ and $y$ is in $I$. Now, we define a graded space $\mathcal{D}:=D_0\oplus D_1$ with $D_0=D$ and $D_1=I$. 
\bigskip

$\blacktriangleright$ First, define a degree $-1$ map $\mu_1:I\rightarrow D$ by the inclusion map.\\
 
$\blacktriangleright$ Also, consider the map $\mu_2:\otimes^2 D\rightarrow D$ defined by 
\begin{equation}\label{mu2}
\mu_2(x,y)=\frac{1}{2}(x\dashv y+ x\vdash y), \quad \forall x,y\in D.
\end{equation}
Then, $\mu_2(x,y)\in I$ if either of the two elements $x$ and $y$ is in $I$. More preciesely, for $x\in I$ and $y,w\in D$, we have
\begin{align*}
\mu_2(x,y)\vdash w&=\frac{1}{2}(x\dashv y+ x\vdash y)\vdash w =(x\vdash y)\vdash w=0.\quad (\mbox{since } x\in I)\\
w\dashv \mu_2(x,y)&=\frac{1}{2}\big(w\dashv(x\dashv y+ x\vdash y)\big)=w\dashv(x\vdash y)=0. \quad (\mbox{since } x\in I)
\end{align*}
So, $\mu_2(x,y)\in I,$ for all $x\in I$ and $y\in D$. Similarly, for all $x\in D$ and $y\in I$, we have $\mu_2(x,y)\in D$. In turn, it follows that $\mu_2$ induces a degre $0$ map $\mu_2:\mathcal{D}\otimes \mathcal{D}\rightarrow \mathcal{D}$.

\bigskip

$\blacktriangleright$ Finally, we consider the map $\mu_3:\otimes^3 D\rightarrow D$ defined by 
\begin{equation}\label{mu3}
\mu_3(x,y,z)=\mu_2(\mu_2(x,y),z)-\mu_2(x,\mu_2(y,z)), \quad  \forall x,y,z\in D.
\end{equation}
Here, $\mu_3(x,y,z)$ is the associater with respect to the map $\mu_2$. It follows from \eqref{mu2}-\eqref{mu3} that 
\begin{equation}\label{mu3-reduced}
\mu_3(x,y,z)=\frac{1}{4}\big((x\dashv y)\vdash z -x\dashv (y\vdash z)\big), \quad  \forall x,y,z\in D.
\end{equation}

Let us observe that 
 $$w\dashv \mu_3(x,y,z)=0= \mu_3(x,y,z)\vdash w,\quad  \forall x,y,z,w\in D.$$ 
Therefore, $\mu_3(x,y,z)\in I$ for any $x,y,z\in D$.

\begin{Prop}\label{As-2-algebra}
Let $(D,\dashv,\vdash)$ be a dialgebra. Then $(\mathcal{D}=D\oplus I,\mu_1,\mu_2,\mu_3)$ is an associative $2$-algebra, where $\mu_1:I \rightarrow D$ is the inclusion map and the maps $\mu_2$, $\mu_3$ are given by equations \eqref{mu2} and \eqref{mu3-reduced}.
\end{Prop}

\begin{proof}
It is straightforward to check that the identities $(i)-(vi)$ in the definition of an associative $2$-algebra are satisfied for $(\mathcal{D}=D\oplus I,\mu_1,\mu_2,\mu_3)$. We verify the identity $(vii)$, i.e. 
$$\mu_3(\mu_2(x,y),z,w)-\mu_3(x,\mu_2(y,z),w)+\mu_3(x,y,\mu_2(z,w))=\mu_2(\mu_3(x,y,z),w)+\mu_2(x,\mu_3(y,z,w)).$$

To show that the identity holds, let us rewrite the first term of the Left hand side using equations \eqref{mu2} and \eqref{mu3-reduced} 
\begin{align*}
\mu_3(\mu_2(x,y),z,w)=&~~\frac{1}{4}\big((\mu_2(x,y)\dashv z)\vdash w -\mu_2(x,y)\dashv (z\vdash w)\big)\\
=&~~\frac{1}{8}\big(((x\vdash y)\dashv z)\vdash w+((x\dashv y)\dashv z)\vdash w-(x\vdash y)\dashv (z\vdash w)-(x\dashv y)\dashv (z\vdash w)\big)
\end{align*}

Similarly, second and third terms of the left hand side can be written as 
\begin{align*}
\mu_3(x,\mu_2(y,z),w)
=&~~\frac{1}{8}\big((x\dashv (y\vdash z))\vdash w+(x\dashv (y\dashv z))\vdash w-x\dashv ((y\vdash z)\vdash w)-x\dashv ((y\dashv z)\vdash w)\big)\\
\mu_3(x,y,\mu_2(z,w))
=&~~\frac{1}{8}\big((x\dashv y)\vdash (z\vdash w)+(x\dashv y)\vdash (z\dashv w)-x\dashv (y\vdash (z\vdash w))-x\dashv (y\vdash (z\dashv w))\big)
\end{align*}
Simplifying the left hand side, one obtains the following expression
\begin{align*}
&\mu_3(\mu_2(x,y),z,w)-\mu_3(x,\mu_2(y,z),w)+\mu_3(x,y,\mu_2(z,w))\\
=&~~\frac{1}{8}\big(((x\vdash y)\dashv z)\vdash w-(x\vdash y)\dashv (z\vdash w)+(x\dashv y)\vdash (z\dashv w)-x\dashv (y\vdash (z\dashv w))\big)\\
=&~~\frac{1}{8}\big(x\vdash ((y\dashv z)\vdash w)-x\vdash (y\dashv (z\vdash w))+((x\dashv y)\vdash z)\dashv w-(x\dashv (y\vdash z))\dashv w\big)\\
=&~~   \mu_2(x,\mu_3(y,z,w))+\mu_2(\mu_3(x,y,z),w).
\end{align*}
Hence, it follows that $(\mathcal{D}=D\oplus I,\mu_1,\mu_2,\mu_3)$ is an associative $2$-algebra.
\end{proof}

\section{Poisson dialgebras}\label{sec-4}

The notion of Poisson dialgebras was first introduced by Loday in \cite{loday} as a generalization of Poisson algebras \cite{Akman, YKS, Kubo}. We observe that the following definition is more appropriate one with several supporting examples.

\begin{Def}\label{Pois-di}
A Poisson dialgebra $\mathcal{P}$ is a $\mathbb{K}$-vector space with dialgebra products $\dashv$ and $\vdash$, and a Leibniz bracket $[-,-]$ such that the following compatibility relations hold.
\begin{equation}\label{comp-ax1}
[x, y \dashv z] = y\vdash [x, z]+[x,y]\dashv z=[x,y\vdash z],
\end{equation}
\vspace{-.8cm}
\begin{equation}\label{comp-ax2}
[x\dashv y, z] = x\dashv [y, z]+[x,z]\dashv y,\end{equation}
\vspace{-.8cm}
\begin{equation}\label{comp-ax3}
[x\vdash y, z] = x\vdash  [y, z]+[x,z]\vdash  y,
\end{equation}
\vspace{-.8cm}
\begin{equation}\label{comp-ax4}
[x , y]\vdash z = -[y,x]\vdash z,
\end{equation}
\vspace{-.8cm}
\begin{equation}\label{comp-ax5}
x\dashv [y,z]= -x\dashv [z,y],
\end{equation}
for all $x,y,z\in \mathcal{P}$. 
\end{Def}
\noindent We call the last two identities as mixed skew-symmetry of the Leibniz bracket $[-,-]$.

Let $\mathcal{P}$ and $\mathcal{P}^\prime$ be two Poisson dialgebras, then a map $\varphi:\mathcal{P}\rightarrow \mathcal{P}^\prime$ is a homomorphism of Poisson dialgebras if $\varphi$ preserves both the dialgebra products and the Leibniz bracket. We denote the category of Poisson dialgebras by $\mathsf{Pois-di}$.

Note that any Poisson algebra is also a Poisson dialgebra. Thus, we have a functor $\mathsf{inc}:\mathsf{Pois}\rightarrow \mathsf{Pois-di}$ considering a Poisson algebra as a Poisson dialgebra. 

\begin{Exm}\label{1}
Let $(D,\dashv,\vdash)$ be a dialgebra. Then by Proposition \ref{induced-Leibniz}, we have a Leibniz bracket on $D$ defined by
$$[x,y]=x\dashv y-y\vdash x,\quad  \forall x,y\in D.$$
It follows from identities \eqref{ax6}-\eqref{ax8} that the bracket and the dialgebra products satisy the compatibility relations given by the equations \eqref{comp-ax1}-\eqref{comp-ax5}. Thus, we have a Poisson dialgebra structure on any dialgebra.
\end{Exm} 
 
\begin{Exm}\label{Exm2}
Let $P$ be a Poisson algebra and $M$ be a Poisson-bimodule over $P$. If $f: M\rightarrow P$ is a Poisson-bimodule map, then $(M,\dashv,\vdash,[-,-]_\mathsf{M})$ is a Poisson dialgebra, where  
$$m\dashv n=m\cdot f(n),\quad  m\vdash n=f(m)\cdot n,\quad\mbox{and}\quad[m,n]_\mathsf{M}=[m,f(n)], ~~~ \forall m,n\in M.$$
We need to show that compatibility conditions \eqref{comp-ax1}-\eqref{comp-ax5} hold. First, let us consider the following expression.
\begin{align*}
[m, n \dashv p]_\mathsf{M} &=[m,f(n\cdot f(p))]\\
&=[m,f(n)f(p)]\quad\quad\quad\mbox{(since}~ f~ \mbox{is a Poisson-bimodule map)}\\
&=f(n)[m,f(p)]+[m,f(n)]f(p)\quad\mbox{(since}~ M~ \mbox{is a Poisson-bimodule) }\\
&=n\vdash [m,p]_\mathsf{M}+[m,n]_\mathsf{M}\dashv p.
\end{align*}
Since $f:M\rightarrow P$ is a Poisson-bimodule map, we have
$$[m, n \vdash p]_\mathsf{M}=[m, f(f(n)\cdot p)] =n\vdash [m,p]_\mathsf{M}+[m,n]_\mathsf{M}\dashv p=[m, n \dashv p]_\mathsf{M}.$$
Similarly, one can show that
\begin{equation*}
[m\dashv n, p]_\mathsf{M} = m\dashv [n, p]_\mathsf{M}+[m,p]_\mathsf{M}\dashv n,
\end{equation*}
\begin{equation*}
[m\vdash n, p]_\mathsf{M} = m\vdash [n, p]_\mathsf{M}+[m,p]_\mathsf{M}\vdash n,\quad \forall m,n,p\in M.
\end{equation*}
Also, we obtain the following expression by using the fact that $f:M\rightarrow P$ is a Poisson-bimodule map    
\begin{align*}
[m,n]_\mathsf{M}\vdash p = f[m,f(n)]\cdot p = [f(m),f(n)]\cdot p = -[f(n),f(m)]\cdot p= -f[n,f(m)]\cdot p= -[n,m]_{\mathsf{M}}\vdash p.
\end{align*}
Similarly, $$m\dashv [n,p]_{\mathsf{M}}=-m\dashv [p,n]_{\mathsf{M}}, \quad  \forall m,n,p\in P.$$
Thus, we have a Poisson dialgebra structure $(M,\dashv, \vdash,[-,-]_\mathsf{M})$ on $M$.
\end{Exm} 

\begin{Exm}
Let $P$ be a Poisson algebra and $d:P\rightarrow P$ be a linear map. Suppose that $(P,d)$ is a differential Poisson algebra, i.e., 
\begin{align*}
d(xy)&=x(dy)+ (dx) y,\\
d([x,y])&=[dx,y]+[x,dy], \quad \forall x,y\in P
\end{align*}
 and $d^2=0$. Then $\mathcal{P}_d:=(P,\dashv,\vdash,[-,-]_d)$ is a Poisson dialgebra, where the dialgebra products $\dashv, \vdash$, and the Leibniz bracket $[-,-]_d$ are given by 
$$x\dashv y=x d(y),\quad  x\vdash y=d(x) y,\quad\mbox{and} \quad [x,y]_d=[x,dy],\quad\forall x,y\in P.$$
\end{Exm}

\subsection{Averaging operators on Poisson algebras}
Let us recall the notion of averaging operators from \cite{Rota1,Rota2}. Let $P$ be a Poisson algebra with the associative algebra product $\mu$ and the Lie bracket $[-,-]$. A linear map $\alpha: P\rightarrow P $ is called an averaging operator if 
$$\alpha(x\alpha(y))=\alpha(x)\alpha(y)=\alpha(\alpha(x)y)~~\mbox{and}$$
$$[\alpha(x),\alpha(y)]=\alpha([\alpha(x),y]),\quad  \forall x,y\in P.$$
Thus, the map $\alpha$ is an averaging operator on the associative algebra $(P,\mu)$ as well as on the Lie algebra $(P,[-,-])$. A differential $d:P\rightarrow P$ of a Poisson algebra $P$ is a particular case of an averaging operator on $P$.

 In the next result, we construct a Poisson dialgebra from a Poisson algebra endowed with an averaging operator.

\begin{Prop}
Let $P$ be a Poisson algebra with the associative algebra product $\mu$, Lie bracket $[-,-]$, and an averaging operator $\alpha:P\rightarrow P$. Then, $\mathcal{P}=(P,\dashv,\vdash,[-,-]_{\mathcal{P}})$ is a Poisson dialgebra, where 
$$x\dashv y:=x\alpha(y),\quad x\vdash y:=\alpha(x)y,\quad\mbox{and}\quad[x,y]_{\mathcal{P}}=[x,\alpha(y)],~\forall x,y\in P.$$
\end{Prop} 

\begin{proof}
It easily follows that $P$ is a dialgebra with the products $\dashv,\vdash$ and $(P,[-,-]_\mathcal{P})$ is a right Leibniz algebra. We now show that the comaptibility relations \eqref{comp-ax1}-\eqref{comp-ax3} hold. For any $x,y,z\in P$, we have 
\begin{align*}
[x, y \dashv z]_\mathcal{P}&=[x,\alpha(y\alpha(z))]\\
&=[x,\alpha(y)\alpha(z)]\\
&=\alpha(y)[x,\alpha(z)]+[x,\alpha(y)]\alpha(z)\\
&=y\vdash [x,z]_\mathcal{P}+[x,y]_{\mathcal{P}}\dashv z.
\end{align*}
Moreover, $\alpha(y\alpha(z))=\alpha(\alpha(y)z)$ implies that
$$[x, y \dashv z]_\mathcal{P}=y\vdash [x,z]_\mathcal{P}+[x,y]_{\mathcal{P}}\dashv z=[x, y \vdash z]_\mathcal{P}.$$
Similarly, 
\begin{equation*}
[x\dashv y, z]_\mathcal{P} = x\dashv [y, z]_\mathcal{P}+[x,z]_\mathcal{P}\dashv y,
\end{equation*}
\begin{equation*}
[x\vdash y, z]_\mathcal{P} = x\vdash  [y, z]_\mathcal{P}+[x,z]_\mathcal{P}\vdash  y,\quad\forall x,y,z\in P.
\end{equation*}
Next, we consider the following expression
\begin{align*}
[x,y]_{\mathcal{P}}\vdash z=\alpha[x,\alpha(y)]z=[\alpha(x),\alpha(y)]z=-[\alpha(y),\alpha(x)]z=-\alpha[y,\alpha(x)]z=-[y,x]_{\mathcal{P}}\vdash z.
\end{align*}
A similar calculation shows that  
$$x\dashv [y,z]_{\mathcal{P}}=-x\dashv [z,y]_{\mathcal{P}},\quad\forall x,y,z\in P.$$
Hence, $\mathcal{P}=(P,\dashv,\vdash,[-,-]_{\mathcal{P}})$ is a Poisson dialgebra.
\end{proof}

\begin{Def}\label{Poisson-di bimodule} 
Let $\mathcal{P}$ be a Poisson dialgebra. Then a Poisson dialgebra bimodule is a $\mathbb{K}$-vector space $M$ such that it is simultaneously a bimodule over the dialgebra $(\mathcal{P},\dashv, \vdash)$ and a module over the Leibniz algebra $(\mathcal{P},[-,-])$ such that the identities \eqref{comp-ax1}-\eqref{comp-ax3} hold whenever two of the entries are in $\mathcal{P}$ and one entry is in  $M$.
\end{Def}

\begin{Def}
A graded Poisson dialgebra $\mathcal{P}$ of degree $n$ is a $\mathbb{K}$-graded vector space with degree $0$ dialgebra products $\dashv$ and $\vdash$, and a Leibniz bracket $[-,-]$ of degree $-n$ such that 
\begin{equation}\label{graded-comp-ax1}
[x, y \dashv z] = (-1)^{|y|(|x|-n)}y\vdash [x, z]+[x,y]\dashv z=[x,y\vdash z],
\end{equation}
\vspace{-.6cm}
\begin{equation}\label{graded-comp-ax2}
[x\dashv y, z] = x\dashv [y, z]+(-1)^{|y|(|z|-n)}[x,z]\dashv y,\end{equation}
\vspace{-.6cm}
\begin{equation}\label{graded-comp-ax3}
[x\vdash y, z] = x\vdash  [y, z]+(-1)^{|y|(|x|-n)}[x,z]\vdash  y,
\end{equation}
\vspace{-.6cm}
\begin{equation}\label{graded-comp-ax4}
[x , y]\vdash z = -(-1)^{(|x|-n)(|y|-n)}[y,x]\vdash z,
\end{equation}
\vspace{-.6cm}
\begin{equation}\label{graded-comp-ax5}
x\dashv [y,z]= -(-1)^{(|y|-n)(|z|-n)}x\dashv [z,y],
\end{equation}
where $x,y,z\in \mathcal{P}$ and $|r|$ denotes degree of an element $r\in \mathcal{P}$. We call the graded Poisson dialgebra of degree $0$ the Poisson dialgebra structure on a graded vector space, and call the graded Poisson dialgebra of degree $1$ the Gerstenhaber dialgebras.  
\end{Def}


\subsection{Poisson dialgebra associated to a Filtered dialgebra}

A filtration of a dialgebra $(D,\dashv, \vdash)$ is an increasing sequence of subspaces $D_0\subseteq D_1\subseteq D_2\subseteq\ldots$ such that 
$$D=\cup_{n\geq 0} D_n\quad \mbox{and }\quad D_i\dashv D_j,~D_i\vdash D_j\subseteq D_{i+j}, \mbox{for all } i,j\geq 0.$$

There is a graded vector space $\mathrm{Gr}(D)=\oplus_{i\geq 0} \mathrm{Gr}(D)_i$ associated to the filtration of the dialgebra $D$, where $\mathrm{Gr}(D)_0=D_0$ and $\mathrm{Gr}(D)_i=D_{i}/D_{i-1}$ for any $i\geq 1$.
Moreover, this graded vector space $\mathrm{Gr}(D)$ is a graded dialgebra, where the dialgebra products (of degree $0$) are given by 
\begin{align*}
(x+D_{i-1})\dashv (y+D_{j-1})&=x\dashv y+D_{i+j-1},\\
(x+D_{i-1})\vdash (y+D_{j-1})&=x\vdash y+D_{i+j-1},
\end{align*}
for all $x\in D_i$ and $y\in D_j$. 

Next, let us define a degree $0$ bracket on $\mathrm{Gr}(D)$ as follows
\begin{equation}\label{graded Loday Bracket}
[x+ D_{i-1}, y+ D_{j-1}]=x\dashv y - (-1)^{ij}y\vdash x+D_{i+j-1}, \forall x\in D_i, ~y\in D_j.
\end{equation}

We claim that the above bracket is a graded Leibniz bracket of degree $0$. To prove our claim, we consider the following expressions.
\begin{align}\label{LHS}
[[x+D_{i-1},y+D_{j-1}],z+D_{k-1}]
=&[x\dashv y-(-1)^{ij} y\vdash x+D_{i+j-1},z+D_{k-1}]\\\nonumber
=&(x\dashv y-(-1)^{ij} y\vdash x)\dashv z -(-1)^{k(i+j)}z\vdash(x\dashv y-(-1)^{ij} y\vdash x).
\end{align}
\begin{align}\label{RHS-I}
&(-1)^{jk}[[x+D_{i-1},z+D_{k-1}],y+D_{j-1}]\\\nonumber
=&(-1)^{jk}[x\dashv z-(-1)^{ik} z\vdash x+D_{i+k-1},y+D_{j-1}]\\\nonumber
=&(-1)^{jk}\{(x\dashv z-(-1)^{ik} z\vdash x)\dashv y -(-1)^{j(i+k)}y\vdash(x\dashv z-(-1)^{ik} z\vdash x)\}.
\end{align}
\begin{align}\label{RHS-II}
[x+D_{i-1},[y+D_{j-1},z+D_{k-1}]]
=&[x+D_{i-1},y\dashv z-(-1)^{jk} z\vdash y+D_{j+k-1}]\\\nonumber
=&x\dashv (y\dashv z-(-1)^{jk} z\vdash y) -(-1)^{i(j+k)}(y\dashv z-(-1)^{jk} z\vdash y)\vdash x.
\end{align}

From expressions \eqref{LHS}-\eqref{RHS-II} and the identities \eqref{ax1}-\eqref{ax5} it follows that
$$[[x+D_{i-1},y+D_{j-1}],z+D_{k-1}]=(-1)^{jk}[[x+D_{i-1},z+D_{k-1}],y+D_{j-1}]+[x+D_{i-1},[y+D_{j-1},z+D_{k-1}]].$$

Hence, the bracket defined by the equation \eqref{graded Loday Bracket} is a graded Leibniz bracket of degree $0$. Let us denote $\bar{x}_i:=x+D_{i-1},\bar{y}_j:=y+D_{j-1},$ and $\bar{z}_k:=z+D_{k-1}$. Then, we can verify the following identities
\begin{equation}\label{quot-graded-comp-ax1}
[\bar{x}_i, \bar{y}_j \dashv \bar{z}_k] = (-1)^{ij}\bar{y}_j\vdash [\bar{x}_i, \bar{z}_k]+[\bar{x}_i,\bar{y}_j]\dashv \bar{z}_k=[\bar{x}_i,\bar{y}_j\vdash \bar{z}_k],
\end{equation} 
\vspace{-.6cm}
\begin{equation}\label{quot-graded-comp-ax2}
[\bar{x}_i\dashv \bar{y}_j, \bar{z}_k] = \bar{x}_i\dashv [\bar{y}_j, \bar{z}_k]+(-1)^{jk}[\bar{x}_i,\bar{z}_k]\dashv \bar{y}_j,
\end{equation}
\vspace{-.6cm}
\begin{equation}\label{quot-graded-comp-ax3}
[\bar{x}_i\vdash \bar{y}_j, \bar{z}_k] = \bar{x}_i\vdash [\bar{y}_j, \bar{z}_k]+(-1)^{jk}[\bar{x}_i,\bar{z}_k]\vdash \bar{y}_j,
\end{equation}
\vspace{-.6cm}
\begin{equation}\label{quot-graded-comp-ax4}
[\bar{x}_i , \bar{y}_j]\vdash \bar{z}_k = -(-1)^{ij}[\bar{y}_j,\bar{x}_i]\vdash \bar{z}_k,
\end{equation}
\vspace{-.6cm}
\begin{equation}\label{quot-graded-comp-ax5}
\bar{x}_i\dashv [\bar{y}_j,\bar{z}_k]= -(-1)^{jk}\bar{x}_i\dashv [\bar{z}_k,\bar{y}_j].
\end{equation}

Note that the above identities \eqref{quot-graded-comp-ax1}-\eqref{quot-graded-comp-ax5} follows from the fact that any dialgebra carries a Poisson dialgebra structure where the Leibniz bracket is given by the Proposition \ref{induced-Leibniz} (cf. Example \ref{1}). 

\subsection{Gerstenhaber dialgebra associated to a Filtered dialgebra}
Let $(D,\dashv, \vdash)$ be a filtered dialgebra with the filtration $D=\cup_{n\geq 0} D_n$ and $D_0\subseteq D_1\subseteq D_2\subseteq\ldots$ (increasing sequence of subspaces). Let the associated dialgebra structure on the graded vector space $\mathrm{Gr}(D)$ is commutative, i.e.,
$$(x+D_{i-1})\dashv(y+D_{j-1})=(y+D_{j-1})\vdash(x+D_{i-1}),\quad\forall x\in D_i, y\in D_j.$$  
The commuativity condition implies that $x\dashv y-y\vdash x\in D_{i+j-1}$. Thus, the expression 
$$[x+ D_{i-1}, y+ D_{j-1}]=x\dashv y - (-1)^{(i-1)(j-1)}y\vdash x+D_{i+j-2}, \quad\forall x\in D_i, ~y\in D_j$$
defines a degree $-1$ bracket on the graded vector space $\mathrm{Gr}(D)$. A similar calculation as in the previous subsection shows that the above bracket is a degree $-1$ Leibniz bracket on $\mathrm{Gr}(D)$. Furthermore, $\mathrm{Gr}(D)$ is a graded Poisson dialgebra of degree $1$.  

\subsection{Poisson algebra associated to a Poisson dialgebra}

Let $\mathcal{P}$ be a Poisson dialgebra with dialgebra products $\dashv, \vdash$ and Leibniz bracket $[-,-]$. We consider the quotient space of $\mathcal{P}$ by the subspace spanned by the elements of the form $x\dashv y-x\vdash y$ and $[x,y]-[y,x]$ for all $x,y\in \mathcal{P}$. We denote this quotient space by $\mathcal{P}_{\mathrm{Poiss}}$. The left and right dialgebra products are same in the quotient space $\mathcal{P}_{\mathrm{Poiss}}$ and the Leibniz bracket induces a Lie bracket on $\mathcal{P}_{\mathrm{Poiss}}$. Thus, the quotient space $\mathcal{P}_{\mathrm{Poiss}}$ becomes a Poisson algebra. 

This construction gives a functor $\mathrm{(-)_{Poiss}}: \mathsf{Pois-di}\rightarrow \mathsf{Pois}$. We call this functor the Poissonization functor. It is left adjoint to the functor $\mathsf{inc}:\mathsf{Pois}\rightarrow \mathsf{Pois-di}$.

\subsection{Poisson algebra objects in the category $\mathcal{LM}$}
\begin{Def}
A Poisson algebra object in the category $\mathcal{LM}$ is an object $M\xrightarrow{f} P$ with two morphisms
$$\bar{\mu}=(\mu^\prime,\mu):(M\xrightarrow{f} P)\otimes_{\mathcal{LM}}(M\xrightarrow{f} P)\rightarrow (M\xrightarrow{f} P)~~\mbox{and}$$
$$\bar{\nu}=(\nu^\prime,\nu):(M\xrightarrow{f} P)\otimes_{\mathcal{LM}}(M\xrightarrow{f} P)\rightarrow (M\xrightarrow{f} P)$$
such that  
\begin{equation}\label{Pois-1}
\bar{\mu}(\mathsf{1}\otimes_{\mathcal{LM}}\bar{\mu})=\bar{\mu}(\bar{\mu}\otimes_{\mathcal{LM}}\mathsf{1}),
\end{equation}
\begin{equation}\label{Pois-2}
\bar{\nu}(1~2)=-\bar{\nu},
\end{equation}
\begin{equation}\label{Pois-3}
\bar{\nu}(\mathsf{1}\otimes_{\mathcal{LM}}\bar{\nu})+
\bar{\nu}(\mathsf{1}\otimes_{\mathcal{LM}}\bar{\nu})(1~2~3)+\bar{\nu}(\mathsf{1}\otimes_{\mathcal{LM}}\bar{\nu})(1~3~2)=0, 
\end{equation}
\begin{equation}\label{Pois-4}
\bar{\nu}(\mathsf{1}\otimes_{\mathcal{LM}}\bar{\mu})=
\bar{\mu}(\mathsf{1}\otimes_{\mathcal{LM}}\bar{\nu})(1~2)+\bar{\mu}(\bar{\nu}\otimes_{\mathcal{LM}}\mathsf{1}).
\end{equation}
Here, $(1~2),~(1~2~3),~(1~3~2)$ are cycles in symmetric group $S_3$. 
\end{Def}
The first condition \eqref{Pois-1} implies that $M\xrightarrow{f} P $ is an associative algebra object in $\mathcal{LM}$. The identities \eqref{Pois-2} and \eqref{Pois-3} implies that $M\xrightarrow{f} P $ is a Lie algebra object in $\mathcal{LM}$. The last identity \eqref{Pois-4} is the compatibility condition between the two structures in $\mathcal{LM}$.

\begin{Prop}\label{Poisson bimod map-Pdi}
A Poisson algebra object $M\xrightarrow{f} P$ in the category $\mathcal{LM}$ is equivalent to a Poisson-bimodule map $f:M\rightarrow P$, where $P$ is a Poisson algebra and $M$ is a $P$-bimodule. 
\end{Prop}

\begin{proof}
The downstairs of the above conditions \eqref{Pois-1}-\eqref{Pois-3} is equivalent to the fact that $P$ is an associative algebra with the product $\mu$ and a Lie algebra with the bracket $\nu$. The downstairs of the identity \eqref{Pois-4} is equivalent to the statement that the Lie bracket $\nu$ acts as a derivation of the associative product $\mu$. Thus, $P$ is a Poisson algebra.

The upstairs of the identity \eqref{Pois-1} implies that there is an $A$-bimodule structure on $M$ (given by the map $\mu^\prime$). Moreover, the upstairs of the identities \eqref{Pois-2}-\eqref{Pois-3} yields a Lie module structure on $M$ (via the map $\nu^\prime$). In the end, the upstairs of the equation \eqref{Pois-4} implies that the Poisson identity \eqref{Leibniz rule} holds whenever two of the entries are in $P$ and one of them is in $M$. So, $M$ is a Poisson-bimodule over $P$.

The map $\bar{\mu}$ and $\bar{\nu}$ are morphisms in $\mathcal{LM}$, which equivalently means that
$$f(\mu^\prime(m, x))=\mu(f(m), x),~\quad f(\mu^\prime(x, m))=\mu(x, f(m)),\quad \mbox{and}\quad f(\nu^\prime(x,m))=\nu(x,f(m)),\quad\forall m\in M, x\in P.$$

Hence, Poisson algebra object $M\xrightarrow{f} P$ in the category $\mathcal{LM}$ is equivalent to a Poisson-bimodule map $f:M\rightarrow P$.

\end{proof}

\begin{Thm}
For a Poisson algebra object $M\xrightarrow{f} P$ in $\mathcal{LM}$, we have a Poisson dialgebra structure on $M$. On the other hand, for any Poisson dialgebra $\mathcal{P}$ the projection map $\mathcal{P}\xrightarrow{p}\mathcal{P}_{\mathrm{Poiss}}$ is a Poisson algebra object in $\mathcal{LM}$. The two functors $\mathsf{Pois-di}\longleftrightarrow \mathsf{Pois}_{\mathcal{LM}}$ are adjoint to each-other. 

\end{Thm}

\begin{proof}
Let $M\xrightarrow{f} P$ be a Poisson algebra object in $\mathcal{LM}$. Proposition \ref{Poisson bimod map-Pdi} implies that the Poisson object $M\xrightarrow{f} P$ is equivalent to a Poisson-bimodule map $f:M\rightarrow P$, where $P$ is a Poisson algebra and $M$ is a $P$-bimodule. By Example \ref{Exm2}, we have a Poisson dialgebra structure $(M,\dashv, \vdash,[-,-]_\mathsf{M})$ on $M$ with the dialgebra products
$$m\dashv n=m\cdot f(n)\quad \mbox{and}\quad m\vdash n=f(m)\cdot n, \quad \forall m,n\in M,$$
and the (right) Leibniz bracket given by 
$$[m,n]_\mathsf{M}=[m,f(n)],\quad\forall m,n\in M.$$
Here, the bracket $[-,-]$ on the right hand side is the Lie action of $P$ on $M$. We denote by $\mathsf{F}: \mathsf{Pois}_{\mathcal{LM}}\rightarrow \mathsf{Pois-di}$, the functor that associates the Poisson dialgebra $(M,\dashv, \vdash,[-,-]_\mathsf{M})$ to a Poisson object 
$M\xrightarrow{f} P$ in $\mathcal{LM}$. Also, denote $\mathsf{G}:\mathsf{Pois-di}\rightarrow \mathsf{Pois}_{\mathcal{LM}}$, the functor associating the Poisson object $\mathcal{P}\xrightarrow{p}\mathcal{P}_{\mathrm{Poiss}}$ to a Poisson dialgebra $\mathcal{P}$.

Next, we need to show that there is a bijection between the sets of homomorphisms
$$\mathrm{Hom}_{\mathsf{Pois}_{\mathcal{LM}}}(\mathcal{P}\xrightarrow{p} \mathcal{P}_{\mathrm{Poiss}},M\xrightarrow{f} P)\cong \mathrm{Hom}_\mathsf{Pois-di}(\mathcal{P},M).$$
Since the Poissonization functor $\mathrm{(-)_{Poiss}}: \mathsf{Pois-di}\rightarrow \mathsf{Pois}$ is left adjoint to the functor $\mathsf{inc}:\mathsf{Pois}\rightarrow \mathsf{Pois-di}$, the bijection follows from a similar argument as in the proof of Theorem \ref{As_LM}.

 Moreover, the above bijection is functorial in $\mathcal{P}\in \mathsf{Pois-di}$ and $M\xrightarrow{f}P\in \mathsf{Pois}_{\mathcal{LM}}$ since for any morphism $\alpha:\mathcal{P}\rightarrow \mathcal{P}^\prime$ in $\mathsf{Pois-di}$ and $\beta:(M\xrightarrow{f}P)\rightarrow (M^\prime\xrightarrow{f^\prime}P^\prime)$ in $\mathsf{Pois}_{\mathcal{LM}}$, we have the following commutative diagram:

$$\begin{CD}
\mathrm{Hom}_{\mathsf{Pois}_{\mathcal{LM}}}(\mathsf{G}(\mathcal{P}^{\prime}), M\xrightarrow{f}P) @> \mathsf{G}(\alpha)^* >>
\mathrm{Hom}_{\mathsf{Pois}_{\mathcal{LM}}}(\mathsf{G}(\mathcal{P}), M\xrightarrow{f}P) @> \beta_{*} >> \mathrm{Hom}_{\mathsf{Pois}_{\mathcal{LM}}}(\mathsf{G}(\mathcal{P}), M^{\prime}\xrightarrow{f^{\prime}}P^{\prime})\\
@V\cong VV @V\cong VV @V\cong VV\\
\mathrm{Hom}_{\mathsf{Pois-di}}(\mathcal{P}^{\prime},\mathsf{F} (M\xrightarrow{f}P)) @> \alpha^* >>
\mathrm{Hom}_{\mathsf{Pois-di}}(\mathcal{P}, \mathsf{F}(M\xrightarrow{f}P)) @> \mathsf{F}(\beta)_{*} >> \mathrm{Hom}_{\mathsf{Pois-di}}(\mathcal{P}, \mathsf{F}(M^{\prime}\xrightarrow{f^{\prime}}P^{\prime})),
\end{CD}$$
where $\alpha^*$ and $\beta_*$ are the maps induced by $\alpha$ and $\beta$. 

\end{proof}

\section{Homotopy algebra structures associated to Poisson dialgebras}\label{sec-5}

In Section $3$, we discussed that there is an associative $2$-algebra $(D\oplus I, \mu_1,\mu_2,\mu_3)$ associated to a dialgebra $D$. In this section, we show that there is a graded vector space associated to a Poisson dialgebra, which carries both a Lie $2$-algebra structure and an associative $2$-algebra structure. In the end, we consider the particular case when the left and right product of the Poisson dialgebra becomes identical, and we obtain an associated homotopy Poisson algebra of finite type. We first recall the definition of a Lie $2$-algebra.

\begin{Def}[\cite{Baez}]
A graded vector space $\mathcal{G}=\mathfrak{g}_{0}\oplus \mathfrak{g}_{1}$ is called a Lie $2$-algebra, if $\mathcal{G}$ is equipped with linear maps ${l_i:\wedge^i\mathcal{G}\rightarrow \mathcal{G}}$  of degrees $i-2$, for $i=1,2,3$ and the following identities hold.
\begin{enumerate}[(i)]
\item $l_1 l_2(x,a)=l_2(x,l_1(a))$,
\item $l_2(l_1(a),b)=l_2(a,l_1(b))$,
\item $l_2(x,l_2(y,z))+l_2(y,l_2(z,x))+l_2(z,l_2(x,y))=l_1l_3(x,y,z)$,
\item $l_2(x,l_2(y,a))+l_2(y,l_2(a,x))+l_2(a,l_2(x,y))=l_3(x,y,l_1(a))$,
\item $l_3(l_2(x,y),z,w)+l_3(l_2(y,z),w,x)+l_3(l_2(z,w),x,y)+l_3(l_2(w,x),y,z)=l_2(l_3(x,y,z),w)+l_2(l_3(y,z,w),x)+l_2(l_3(z,w,x),y)$,
\end{enumerate}
for all $a,b\in \mathfrak{g}_0$ and $x,y,z,w\in \mathfrak{g}_1$. 
\end{Def}

Let $\mathfrak{g}$ be a Leibniz algebra and the set 
$$Z(\mathfrak{g})=\{x\in \mathfrak{g}~|~[y,x]=0,\quad\forall y\in \mathfrak{g}\}$$
denotes the right center  of the Leibniz algebra $\mathfrak{g}$. We consider the graded vector space $\mathfrak{g}\oplus Z(\mathfrak{g})$, where $\mathfrak{g}$ is degree $0$ and the center is of degree $1$. 
Then let us recall the following result from \cite{Sheng}.

\begin{Prop}\label{Lie-2-algebra}
The tuple $(\mathfrak{g}\oplus Z(\mathfrak{g}), l_1, l_2, l_3)$ is a Lie $2$-algebra, where the map $l_1: Z(\mathfrak{g})\rightarrow \mathfrak{g}$ is the inclusion map and the maps $l_2$ and $l_3$ are given as follows:
\begin{align*}
l_2(x,y)&=\frac{1}{2}([x,y]-[y,x]), \quad\forall x,y\in \mathfrak{g},\\
l_3(x,y,z)&= \frac{1}{4}\big( [[z,y],x]+[[x,z],y]+[[y,x],z]\big), \quad\forall x,y,z\in \mathfrak{g}
\end{align*}
 \end{Prop}

Let $\mathcal{P}:=(P,\dashv, \vdash,[-,-])$ be a Poisson dialgebra. We have seen that the set  
$$I=\{x\in P~|~x\vdash y=0=y\dashv x,\quad \forall y\in P\}$$
becomes an ideal in the underlying dialgebra structure $(P,\dashv, \vdash)$. Denote by $Z(P)$, the right center of the underlying Leibniz algebra structure $(P,[-,-])$ of $\mathcal{P}$. Say, $J$ be the intersection of the two ideals $I$ and  $Z(P)$. Then, we associate a graded vector space $P\oplus J$ associated to the Poisson dialgebra $\mathcal{P}$, where $P$ is degree $0$ and $J$ is degree $1$. With the above notations, we have the following theorem.

\begin{Thm}
For a Poisson dialgebra $\mathcal{P}:=(P,\dashv, \vdash,[-,-])$, the associated graded vector space $P\oplus J$ has the following homotopy algebra structures.
\begin{enumerate}
\item $(P\oplus J,\mu_1,\mu_2,\mu_3)$ becomes an associative $2$-algebra, where $\mu_1:J \rightarrow P$ is the inclusion map and the maps $\mu_2$, $\mu_3$ are given by equations \eqref{mu2} and \eqref{mu3-reduced}.

\item $(P\oplus J,l_1,l_2,l_3)$ is a Lie $2$-algebra, where $l_1:J \rightarrow P$ is the inclusion map and the maps $l_2$, $l_3$ are given as in Proposition \ref{Lie-2-algebra}.
\end{enumerate}

\end{Thm}

\begin{proof}

$(1).$ To prove the first statement, we only need to show that 
$$\mu_2(x,c), \mu_2(c,x)\in J, \quad \forall c\in J, x\in P$$ and $$\mu_3(x,y,z)\in J, \quad\forall x,y,z\in P.$$
By the definition of the map $\mu_2$, we have $$\mu_2(x,y)=\frac{1}{2}(x\dashv y+ x\vdash y),\quad \forall x,y\in P.$$ 
For any $c\in J=I\cap Z(P)$,   
$$[z,\mu_2(x,c)]=\frac{1}{2}[z,x\vdash c]=x\vdash[z, c]+[z,x]\dashv c=0,\quad \forall z,x\in P.$$
Similarly, $[z,\mu_2(c,x)]=0$ for all $z,x\in P$. Therefore, 
$$\mu_2(x,c), \mu_2(c,x)\in J, \quad \forall c\in J, x\in P$$
Next, for all $x,y,z,w\in P$, we have the following expression 
\begin{align*}
&[w, \mu_3(x,y,z)]\\
=&\frac{1}{4}[w,(x\dashv y)\vdash z-x\dashv (y\vdash z)]\\
=&\frac{1}{4}\Big((x\dashv y)\vdash[w,z]+[w,x\dashv y]\dashv z-[w,x]\dashv (y\vdash z)-x\vdash [w,y\vdash z]\Big)\\
=&\frac{1}{4}\Big((x\dashv y)\vdash[w,z]+([w,x]\dashv y)\dashv z+(x\vdash [w,y])\dashv z-[w,x]\dashv (y\vdash z)-x\vdash (y\vdash[w, z])-x\vdash ([w,y]\dashv z)\Big)\\
=&0.
\end{align*}
It implies that $\mu_3(x,y,z)\in Z(P)$ for all $x,y,z \in P$. By Proposition \ref{As-2-algebra}, we know that $\mu_3(x,y,z)\in I$ for all $x,y,z \in P$. Hence, 
$$\mu_3(x,y,z)\in J, \quad\forall x,y,z\in P.$$

\bigskip

\noindent $(2).$ Again, to prove the second statement, we need to  simply verify that 
$$l_2(x,c)\in J, \quad l_3(x,y,z)\in J\quad \forall c\in J,~~  x,y,z\in P.$$

For any $c\in J$, we have the following expression 
$$l_2(c,x)\vdash z=\frac{1}{2}[c,x]\vdash z=[c\vdash z,x]-c\vdash [z,x]=0$$
Also,
$$z \dashv l_2(c,x)=\frac{1}{2} z\dashv [c,x]=[z\dashv c,x]-[z,x]\dashv c=0$$
Thus, it is clear that $l_2(x,y)\in J$ whenever atleast one of the elements $x$ and $y$ is in $J$.

Now, by the definition of the map $l_3$ and mixed skew-symmetry of the bracket $[-,-]$, we get  
\begin{align*}
l_3(x,y,z)\vdash w&= \frac{1}{4}\big( [[z,y],x]+[[x,z],y]+[[y,x],z]\big)\vdash w\\
&= \frac{1}{4}\big( [[z,y],x]-[[z,x],y]-[z,[y,x]]\big)\vdash w \\
&=0,\quad\quad \forall x,y,z, w \in \mathfrak{g}.
\end{align*}
Similar calculation shows that $w\dashv l_3(x,y,z)=0$.

\end{proof}

\subsection{Homotopy Poisson algebra associated to a reduced Poisson dialgebra}
In this subsection, we consider the particular case of a Poisson dialgebra $\mathcal{P}:=(P,\dashv, \vdash, [-,-])$ when the left and right dialgebra products become identical, i.e., $$x\dashv y=x\vdash y,\quad  \forall x,y\in P.$$ 
We denote $\mu(x,y):=x\dashv y=x\vdash y$. In this case, we call $\mathcal{P}=(P,\mu,[-,-])$ a reduced Poisson dialgebra.

\begin{Rem}
Let $(P,\mu)$ is a unital associative algebra. With the mixed skew-symmetry property, it follows that the reduced Poisson dialgebra simply becomes a Poisson algebra. Therefore, we only consider non-unital associative algebras in this section.
\end{Rem}

There are two homotopy algebra structures on the graded space $P\oplus J$ for a Poisson dialgebra $\mathcal{P}$. It would be interesting to find a compatibility relation between them, which is subject to  further investigation. However, we observe that if the Poisson dialgebra is reduced, then we can obtain a homotopy Poisson algebra structure of finite type on $P\oplus J$. 
We first recall the definition of a homotopy Poisson algebra from \cite{HomotopyPoiss, Mehta}. These algebras are closely related to homotopy Poisson manifolds \cite{PoissManifold1} and higher Poisson manifolds \cite{HigherPoiss1, HigherPoiss2}.

\begin{Def}
A homotopy Poisson algebra (of degree $0$) is a graded associative algebra $\mathcal{A}$ with an $L_{\infty}$ algebra structure $\{l_k\}_{k\geq 1}$ on $\mathcal{A}$ such that the map $\phi: \mathcal{A}\rightarrow \mathcal{A}$ given by 
$$\phi(x)=l_k(x_1,\ldots, x_{k-1}, x), \quad \mbox{for }x_1,\ldots, x_{k-1},x \in \mathcal{A}$$
is a derivation of degree $\varepsilon:=2-k+\sum_{i=1}^{k-1}|x_i|$, which means that
$$l_k(x_1,\ldots, x_{k-1}, xy)=l_k(x_1,\ldots, x_{k-1}, x)y+(-1)^{\varepsilon.|x| }x l_k(x_1,\ldots, x_{k-1}, y), \quad \forall x,y\in \mathfrak{A}.$$ 
A homotopy Poisson algebra is called of finite type if there exists some $m$ such that for $k>m$, $l_k=0$. A $2$-term homotopy Poisson algebra is a homotopy Poisson algebra structure on a $2$-term graded associative algebra $\mathcal{A}=A_0\oplus A_1$.
\end{Def}

\begin{Prop}
For a reduced Poisson dialgebra $(P,\mu,[-,-])$, the tuple $(P\oplus J,\mu,l_1,l_2,l_3)$ becomes a $2$-term homotopy Poisson algebra (of degree $0$).
\end{Prop}

\begin{proof}
The product $\mu$ is given by $ \mu(x,y):=x\dashv y=x\vdash y$, then the graded vector space $P\oplus J$ becomes a graded associative algebra with the product $\mu$ concentrated in degree $0$. Moreover, the tuple $(P\oplus J,l_1,l_2,l_3)$ is a Lie $2$-algebra, where $l_1:J \rightarrow P$ is the inclusion map and the maps $l_2$, $l_3$ are given as in Proposition \ref{Lie-2-algebra}.

A straightforward calculation shows that for $x,y\in P\oplus J$, the map $z\mapsto l_2(x,z)$ is a derivation of degree $|x|$ and the map $z\mapsto l_3(x,y,z)$ is derivation of degree $|x|+|y|-1$. Hence, the tuple $(P\oplus J,\mu,l_1,l_2,l_3)$ becomes a $2$-term homotopy Poisson algebra (degree $0$). 
\end{proof}

\end{document}